\documentclass{svproc}
\usepackage{url}

\usepackage{amsmath,amssymb,enumerate}
\usepackage{graphicx,multicol,footmisc}
\usepackage[colorlinks=true,linkcolor=blue]{hyperref}

\newtheorem{fact}{Fact}

\renewcommand{\qed}{\hfill $\square$}

\begin{document}
\mainmatter              
\title{On supercompactness of $\omega_1$}
\titlerunning{$\omega_1$ supercompact}  
%
\author{Daisuke Ikegami\inst{1} \and Nam Trang\inst{2}}
\authorrunning{Ikegami and Trang} 
%
\tocauthor{Daisuke Ikegami and Nam Trang}
\institute{SIT Research Laboratories, Shibaura Institute of Technology, \\
3-7-5 Toyosu Koto-ward, Tokyo 135-8548, JAPAN,\\
\email{ikegami@shibaura-it.ac.jp},\\ WWW home page:
\texttt{http://www.sic.shibaura-it.ac.jp/\~{}ikegami/}
\and
Department of Mathematics, University of North Texas,\\
1155 Union Circle \#311430, Denton, TX 76203-5017, USA,\\
\email{Nam.Trang@unt.edu},\\ WWW home page:
\texttt{http://www.math.unt.edu/\~{}ntrang/}
}

\maketitle              

\begin{abstract}
This paper studies structural consequences of supercompactness of $\omega_1$ under $\sf{ZF}$. We show that the Axiom of Dependent Choice $(\sf{DC})$ follows from ``$\omega_1$ is supercompact". ``$\omega_1$ is supercompact" also implies that $\sf{AD}^+$, a strengthening of the Axiom of Determinacy $(\sf{AD})$, is equivalent to $\sf{AD}_\mathbb{R}$. It is shown that ``$\omega_1$ is supercompact" does not imply $\sf{AD}$. The most one can hope for is Suslin co-Suslin determinacy. We show that this follows from ``$\omega_1$ is supercompact" and Hod Pair Capturing $(\sf{HPC})$, an inner-model theoretic hypothesis that imposes certain smallness conditions on the universe of sets. ``$\omega_1$ is supercompact" on its own implies that every Suslin co-Suslin set is the projection of a determined (in fact, homogenously Suslin) set. ``$\omega_1$ is supercompact" also implies all sets in the Chang model have all the usual regularity properties, like Lebesgue measurability and the Baire property.

\medskip

\end{abstract}

\section{Introduction}
Under $\sf{ZFC}$, successor cardinals (like $\omega_1$) are ``small". If $\alpha = \beta^+$ is a successor cardinal, then there is an injection from $\alpha$ into $\mathcal{P}(\beta)$.\footnote{This is equivalent to ``there is a surjection from $\mathcal{P}(\beta)$ onto $\alpha$" under $\sf{ZFC}$. Without the Axiom of Choice, the equivalence can fail.} Without the Axiom of Choice, it is possible for successor cardinals like $\omega_1$ to exhibit large cardinal properties. For instance, it has been known since the 1960's that $\omega_1$ can be measurable under $\sf{ZF}$; this in particular implies that $\omega_1$ is regular and there is no injection of $\omega_1$ into $\mathcal{P}(\omega)$. We believe this result is independently due to Jech (\cite{Jech}) and Takeuti (\cite{Takeuti}). Furthermore, Takeuti, in the same paper \cite{Takeuti}, is able to show that ``$\sf{ZF} + \omega_1$ is supercompact" is consistent relative to ``$\sf{ZFC} + $ there is a supercompact cardinal". The Takeuti model $\mathcal{T}$ in which $\omega_1$ is supercompact is essentially the Solovay's model. Suppose $\sf{ZFC}$ holds and there is a supercompact cardinal. Let $\kappa$ be a supercompact cardinal. Let $g\subset Coll(\omega,<\kappa)$ be $V$-generic for the collapse forcing. Let $\mathbb{R}^* = \mathbb{R}^{V[g]}$. Then Takeuti's model $\mathcal{T}$ is the symmetric model $V(\mathbb{R}^*)$.

Another major development started in the 1960's in set theory concerns the theory of infinite games with perfect information. The Axiom of Determinacy $(\sf{AD})$ asserts that in an infinite game where players take turns to play integers, one of the players has a winning strategy (see the next section for more detailed discussions on $\sf{AD}$ and its variations). It is well-known that $\sf{AD}$ contradicts the Axiom of Choice. Solovay has shown that $\sf{AD}$ implies $\omega_1$ is measurable and $\sf{AD}_\mathbb{R}$ implies that there is a supercompact (countably complete, normal, fine) measure on $\mathcal{P}_{\omega_1}\mathbb{R}$. Structural consequences of $\sf{AD}$ have been extensively investigated, most notably by the Cabal seminar members. Through work of Harrington, Kechris, Neeman, Woodin amongst others, we know that $\omega_1$ is $\alpha$-supercompact for every ordinal $\alpha<\Theta$ under $\sf{AD}^+$, a strengthening of $\sf{AD}$. By \cite{trang2016determinacy}, $\sf{AD}$ and $\sf{AD}_\mathbb{R}$ cannot imply $\omega_1$ is supercompact. Woodin (see below) shows that $\sf{AD}$ is consistent with ``$\omega_1$ is supercompact."

It can be shown that the theories ``$\sf{ZF} + \omega_1$ is measurable" is equiconsistent with ``$\sf{ZFC} + $ there is a measurable cardinal". The question of whether ``$\sf{ZF} + \omega_1$ is supercompact" is equiconsistent with ``$\sf{ZFC} + $ there is a supercompact cardinal" is much more subtle. Woodin, in an unpublished work in the 1990's, is able to show that the former is much weaker than the latter. Woodin's model is a variation of the Chang model. For each $\lambda$, let $\mathcal{F}_\lambda$ be the club filter on $\mathcal{P}_{\omega_1}(\lambda^\omega)$. Woodin's model is defined as
\begin{center}
$\mathcal{C}^+ = \rm{L}$$(\bigcup_{\lambda\in Ord} \lambda^\omega)[\langle\mathcal{F}_\lambda : \lambda\in Ord \rangle]$.\footnote{The Chang model is defined as $ L(\bigcup_{\lambda\in Ord} \lambda^\omega)$.}
\end{center}
Woodin shows that if there is a proper class of Woodin cardinals which are limits of Woodin cardinals, then $\mathcal{C}^+$ satisfies the $\sf{AD}$ and $\omega_1$ is supercompact. We note that in Takeuti's model, $\sf{AD}$ fails.\footnote{This gives a proof that $\sf{AD}$ does not imply ``$\omega_1$ is supercompact".} Suppose not. Letting $A\subset \kappa = \omega_1^\mathcal{T}$ be a set in $V$ that codes $V_\kappa$, then $A\subset \omega_1^{\mathcal{T}}$ is in $\textrm{L}[x]$ for some real $x\in\mathbb{R}^*$. This is a consequence of $\sf{AD}$. However, there is some $\alpha <\kappa$ such that $x\in V[g\restriction\alpha]$, where $g\restriction \alpha$ is the restriction of $g$ to $Coll(\omega,<\alpha)$. In $V[g\restriction\alpha]$, $\kappa$ is a limit of measurable cardinals and $A$ codes $V_\kappa$, so if $A\in \rm{L}$$[x]$, then $\rm{L}$$[x]\vDash$ ``there is a measurable cardinal". This is impossible. 

The theory ``$\omega_1$ is supercompact" and variations of Woodin's model $\mathcal{C}^+$ are intimately related to determinacy theory as well as modern developments in descriptive inner model theory, cf. \cite[Conjecture 1.8]{GammaUB}. The following conjecture captures some of these relationships and is an important test question for the future development of descriptive inner model theory and the core model induction.
\begin{conjecture}
The following theories are equiconsistent.
\begin{enumerate}[(i)]
\item $\sf{ZF} + \omega_1$ is supercompact.
\item $\sf{ZF} + \sf{AD} + \omega_1$ is supercompact.
\item $\sf{ZFC} + $ there is a Woodin cardinal which is a limit of Woodin cardinals.
\end{enumerate}
\end{conjecture}
\cite{trang2016determinacy}, \cite{trang2011determinacy}, \cite{trang2015structure} made some progress in resolving the conjecture by exploring consistency strength and structural consequences of various fragments of supercompactness of $\omega_1$.

This paper studies structural consequences of (full) supercompactness of $\omega_1$ under $\sf{ZF}$. We first show the following basic structural  consequences.

\begin{theorem}\label{thm:choice}
Assume that $\omega_1$ is supercompact. Then
\begin{enumerate}
\item the Axiom of Dependent Choices (\textsf{DC}) holds, while

\item there is no injection from $\omega_1$ to $2^{\omega}$.
\end{enumerate}
\end{theorem}

The useful fact that $\sf{DC}$ holds can be used to derive other determinacy-like consequences such as:

\begin{theorem}\label{thm:weak_homo}
Assume $\omega_1$ is supercompact. Then every tree is weakly homogeneous.
\end{theorem}

\begin{theorem}\label{thm:det}
Assume $\omega_1$ is supercompact and Hod Pair Capturing ($\textsf{HPC}$). Then for any $A$ such that $A$ is Suslin co-Suslin, $A$ is determined. 
\end{theorem}

See Section \ref{sec:HPC} for more detailed discussions on the hypothesis $\sf{HPC}$. Under ``$\omega_1$ is supercompact", we also show that $\sf{AD}^+$ and $\sf{AD}_\mathbb{R}$ are equivalent.

\begin{theorem}\label{thm:equiv}
Assume $\omega_1$ is supercompact. Then the following theories are equivalent:
\begin{enumerate}
\item $\textsf{AD}^+$.
\item $\textsf{AD}_\mathbb{R}$.
\end{enumerate}
\end{theorem}

``$\omega_1$ is supercompact" also implies a large collection of sets of reals are determined (\cite{trang2016determinacy}) and perhaps an even larger collection of sets of reals admit $\infty$-Borel representations.

\begin{theorem}\label{thm:infty-Borel}
Assume that $\omega_1$ is supercompact. Then every subset of $2^{\omega}$ in the Chang model $\mathrm{L}(\bigcup_{\lambda \in \text{Ord}} \lambda^{\omega})$ is $\infty$-Borel.
\end{theorem}

The paper is organized as follows. Section \ref{sec:basic} summarizes basic concepts and definitions used in this paper. In Section \ref{sec:DC}, we prove Theorem \ref{thm:choice}. The proof of Theorem \ref{thm:infty-Borel} is given in Section \ref{sec:Chang model}. In Section \ref{sec:weakhomo}, we prove Theorem \ref{thm:weak_homo}. Section \ref{sec:equiv} introduces the notion of the envelope of a pointclass and proves Theorem \ref{thm:equiv}. Finally, Section \ref{sec:HPC} explains $\sf{HPC}$ and proves Theorem \ref{thm:det}.
\\
\\
\noindent \textbf{Acknowledgement.} We would like to thank Hugh Woodin for communicating his insights on this subject as well as his results concerning the model $\mathcal{C}^+$. The first author would like to thank the Japan Society for the Promotion of Science (JSPS) for its generous support through the grant with JSPS KAKENHI Grant Number 15K17586. The second author would like to thank the National Science Foundation (NSF) for its generous support through Grants DMS-1565808 and DMS-1849295.

\section{Definitions and basic concepts}\label{sec:basic}

Throughout this paper, we work in \textsf{ZF} without the Axiom of Choice. For a nonempty set $A$, the axiom { $\textsf{DC}_A$} states that for any relation $R$ on $A$ such that for any element $x$ of $A$ there is an element $y$ of $A$ with $(x,y) \in R$, there is a function $f \colon \omega \to A$ such that for all natural numbers $n$, $\bigl( f(n) , f(n+1) \bigr) \in R$. The \textbf{Axiom of Dependent Choices} (\textbf{\textsf{DC}}) states that for any nonempty set $A$, $\textsf{DC}_A$ holds.

For a set $X$, $X^{<\omega}$ denotes the set of all finite sequences of elements of $X$, and $X^{\omega}$ denotes the set of all functions from $\omega$ to $X$. In particular, $2^{\omega}$ denotes the set of all function from $\omega$ to $2 = \{ 0, 1\}$, not an ordinal or a cardinal. For a set $X$, we often consider $X^{\omega}$ as a topological space whose basic open sets are of the form $O_s = \{ x \in X^{\omega} \mid s \subseteq x \}$ for $s \in X^{<\omega}$. For a set $X$ and an infinite cardinal $\kappa$, let $\mathcal{P}_{\kappa} X$ be the set of all subsets $\sigma$ of $X$ such that $\sigma$ is well-orderable and its cardinality is less than $\kappa$. 

Let us review some basic terminology on filters. For a set $Z$, a \textbf{filter on $Z$} is a collection of subsets of $Z$ closed under supersets and finite intersections. A filter on $Z$ is \textbf{$\sigma$-complete} if it is closed under countable intersections. A filter on $Z$ is \textbf{non-trivial} if the empty set $\emptyset$ does not belong to the filter. A filter on $Z$ is an \textbf{ultrafilter} (or a \textbf{measure}) if it is non-trivial and for any subset $A$ of $Z$, either $A$ or $Z \setminus A$ is in the filter. Given a formula $\phi$ and an ultrafilter $\mu$ on $Z$, if the set $A = \{ \sigma \in Z \mid \phi (\sigma) \}$ is in $\mu$, then we say \lq\lq for $\mu$-measure one many $\sigma$, $\phi (\sigma)$ holds".

Let us introduce fineness and normality of ultrafilters on $\mathcal{P}_{\kappa} X$. An ultrafilter $\mu$ on $\mathcal{P}_{\kappa} X$ is \textbf{fine} if for any element $x$ of $X$, for $\mu$-measure one many $\sigma$, $x$ is in $\sigma$. An ultrafilter $\mu$ on $\mathcal{P}_{\kappa} X$ is \textbf{normal} if for any set $A$ in $\mu$ and $f \colon A \to \mathcal{P}_{\kappa} X$ with $\emptyset \neq f(\sigma) \subseteq \sigma$ for all $\sigma \in A$, there is an $x_0 \in X$ such that for $\mu$-measure one many $\sigma$ in $A$, $x_0 \in f(\sigma)$. Notice that this definition of normality is equivalent to the closure under diagonal intersections in \textsf{ZF} while it may not be equivalent to the standard definition of normality with regressive functions $f \colon A \to X$ without the axiom of choice. An ultrafilter on $\mathcal{P}_{\kappa} X$ is a \textbf{fine measure on $\mathcal{P}_{\kappa} X$} if it is $\sigma$-complete and fine. A fine measure on $\mathcal{P}_{\kappa} X$ is a \textbf{normal measure on $\mathcal{P}_{\kappa} X$} if it is normal.

We now introduce the main definitions of this paper:
\begin{definition}
Let $\kappa$ be an infinte cardinal.
\begin{enumerate}
\item Let $X$ be a set.
\begin{enumerate}
\item $\kappa$ is \textbf{$X$-strongly compact} if there is a fine measure on $\mathcal{P}_{\kappa} X$.

\item $\kappa$ is \textbf{$X$-supercompact} if there is a normal measure on $\mathcal{P}_{\kappa} X$.
\end{enumerate}

\item $\kappa$ is \textbf{strongly compact} if for any set $X$, $\kappa$ is $X$-strongly compact.

\item $\kappa$ is \textbf{supercompact} if for any set $X$, $\kappa$ is $X$-supercompact.
\end{enumerate}
\end{definition}

We now review basic notions on determinacy axioms. For a nonempty set $X$, the \textbf{Axiom of Determinacy in $X^{\omega}$} ($\mathbf{\textsf{AD}_X}$) states that for any subset $A$ of $X^{\omega}$, in the Gale-Stewart game with the payoff set $A$, one of the players must have a winning strategy. We write {\textsf{AD}} for $\textsf{AD}_X$. The ordinal $\mathbf{\Theta}$ is defined as the supremum of ordinals which are sujrctive images of $\mathbb{R}$. Under \textsf{ZF}+\textsf{AD}, $\Theta$ is very big, e.g., it is a limit of measurable cardinals wile under \textsf{ZFC}, $\Theta$ is equal to the successor cardinal of the continuum $|\mathbb{R}|$. The \textbf{Ordinal Determinacy} states that for any $\lambda < \Theta$, any continuous function $\pi \colon \lambda^{\omega} \to \omega^{\omega}$, in the Gale-Stewart game with the payoff set $\pi^{-1} (A)$, one of the players must have a winning strategy. In particular, Ordinal Determinacy implies AD while it is still open whether the converse holds under \textsf{ZF}+\textsf{DC}. 

We will introduce the notion of $\infty$-Borel codes. Before that, we review some terminology on trees. Given a set $X$, a \textbf{tree on $X$} is a collection of finite sequences of elements of $X$ closed under subsequences. Given an element $t$ of $X^{<\omega}$, $\text{lh}(t)$ denotes its length, i.e., the domain or the cardinality of $t$. Given a tree $T$ on $X$ and elements $s$ and $t$ of $T$, $s$ is an \textbf{immediate successor of $t$ in $T$} if $s$ is a supersequence of $t$ and $\text{lh}(s) = \text{lh}(t) + 1$. Given a tree $T$ on $X$ and an element $t$ of $T$, $\text{Succ}_T (t)$ denotes the collection of all immediate successors of $t$ in $T$. An element $t$ of a tree $T$ on $X$ is \textbf{terminal} if $\text{Succ}_T (t) = \emptyset$. For an element $t$ of a tree $T$ on $X$, $\text{term}(T)$ denotes the collection of all terminal elements of $T$. Given a tree $T$ on $X$,$\mathbf{[T]}$ denotes the collection of all $x \in X^{\omega}$ such that for all natural numbers $n$, $x \upharpoonright n$ is in $T$. A tree $T$ on $X$ is \textbf{well-founded} if $[T] = \emptyset$. We often identify a tree $T$ on $X \times Y$ with a subset of the set $(s,t) \in X^{<\omega} \times Y^{<\omega} \mid \text{lh} (s) = \text{lh} (t) \}$, and $\text{p}[T]$ denotes the collection of all $x \in X^{\omega}$ such that there is a $y \in Y^{\omega}$ with $(x, y) \in [T]$.

\begin{definition}
Let $\lambda$ be a non-zero ordinal.
\begin{enumerate}
\item An \textbf{$\infty$-Borel code in $\lambda^{\omega}$} is a pair $(T, \rho)$ where $T$ is a well-founded tree on some ordinal $\gamma$, and $\rho$ is a function from $\text{term}(T)$ to $\lambda^{<\omega}$.

\item Given an $\infty$-Borel code $c = ( T , \rho)$ in $\lambda^{\omega}$, to each element $t$ of $T$, we assign a subset $B_{c,t}$ of $\lambda^{\omega}$ by induction on $t$ using the well-foundedness of the tree $T$ as follows:
\begin{enumerate}
\item If $t$ is a terminal element of $T$, let $B_{c,t}$ be the basic open set $O_{\rho (t)}$ in $\lambda^{\omega}$. 

\item If $\text{Succ}_T (t)$ is a singleton of the form $\{s \}$, let $B_{c,t}$ be the complement of $B_{c,s}$.

\item If $\text{Succ}_T (t)$ has more than one element, then let $B_{c,t}$ be the union of all set of the form $B_{c,s}$ where $s $ is in $\text{Succ}_T (t)$. 
\end{enumerate}
We write $B_c$ for $B_{c , \emptyset}$.

\item A subset $A$ of $\lambda^{\omega}$ is \textbf{$\infty$-Borel} if there is an $\infty$-Borel code $c$ in $\lambda^{\omega}$ such that $A = B_c$. 
\end{enumerate}
\end{definition}

Usually, we use $\infty$-Borel codes and $\infty$-Borel sets only in the spaces $\omega^{\omega}$ or $2^{\omega}$. We use them for general spaces $\lambda^{\omega}$ in Secion~\ref{sec:Chang model}.

In section~\ref{sec:Chang model}, we will use the following characterization of $\infty$-Borelness in the space $\lambda^{\omega}$:
\begin{fact}\label{fact:infty-Borel-char}
Let $A$ be a subset of $2^{\omega}$ Then the following are equivalent:
\begin{enumerate}
\item $A$ is $\infty$-Borel,

\item $A$ is $\infty$-Borel as a subset of $\lambda^{\omega}$, and

\item for some formula $\phi$ and some set $S$ of ordinals, for all elements $x$ of $\lambda^{\omega}$, $x $ is in $A$ if and only if $\mathrm{L}[S,x] \vDash \lq\lq \phi (S, x)"$.
\end{enumerate}
\end{fact}

We now introduce the axiom $\textsf{AD}^+$, and reivew some notions on Suslin sets. The axiom {$\textsf{AD}^+$} states that (a) $\textsf{DC}_{\mathbb{R}}$ holds, (b) Ordinal Determinacy holds, and (c) every subset of $2^{\omega}$ is $\infty$-Borel. Since $\text{AD}^+$ demands Ordinal Determinacy, $\textsf{AD}^+$ implies $\sf{AD}$ while it is open whether the converse holds in \textsf{ZF}+\textsf{DC}. A subset $A$ of $2^{\omega}$ (ot $\omega^\omega$) is \textbf{Suslin} if there are some ordinal $\lambda$ and a tree $T$ on $2 \times \lambda$ $(\omega\times \lambda$ respectively) such that $A = \text{p} [T]$. $A$ is \textbf{co-Suslin} if the complement of $A$ is Suslin. An infinite cardinal $\lambda$ is a \textbf{Suslin cardinal} if there is a subset $A$ of $2^{\omega}$ ($\omega^\omega$) such that there is a tree on $2 \times \lambda$ ($\omega\times \lambda$) such that $A = \text{p}[T]$ while there are no $\gamma < \lambda$ and a tree $S$ on $2 \times \gamma$ ($\omega\times\lambda$) such that $A = \text{p}[S]$. Under \textsf{ZF}+$\textsf{DC}_{\mathbb{R}}$, $\textsf{AD}^+$ is equivalent to the assertion that Suslin cardinals are closed below $\Theta$ in the order topology of $(\Theta , <)$.

\section{Choice principles and supercompactness of $\omega_1$}\label{sec:DC}

In this section, we prove Theorem \ref{thm:choice}.

\begin{proof}[Theorem \ref{thm:choice}]
1. Let $A$ be any nonempty set and $R$ be any relation on $A$ such that for any $x \in A$ there is a $y \in A$ such that $(x,y) \in R$. We will show that there is a function $f \colon \omega \to A$ such that for all natural numbers $n$, $\bigl( f(n) , f(n+1) \bigr) \in R$. 

Since $\omega_1$ is supercompact, there is a fine normal measure on $\mathcal{P}_{\omega_1} A$. We fix such a measure $\mu$. 

\begin{claim}\label{elementarity}
For $\mu$-measure one many elements $\sigma$ of $\mathcal{P}_{\omega_1} A$, the following holds:
\begin{align*}
(\forall x \in \sigma ) \ ( \exists y \in \sigma ) \ ( x,y) \in R
\end{align*}
\end{claim}

\begin{proof}[Claim~\ref{elementarity}]
Suppose not. We will derive a contradiction using $\mu$. Since $\mu$ is an ultrafilter on $\mathcal{P}_{\omega_1} A$, for $\mu$-measure one many elements $\sigma$ of $\mathcal{P}_{\omega_1} A$, the following holds:
\begin{align*}
(\exists x \in \sigma ) \ ( \forall y \in \sigma) \ (x,y) \notin R
\end{align*}
By normality of $\mu$, there is an $x_0 \in A$ such that for $\mu$-measure one many elements $\sigma$ of $\mathcal{P}_{\omega_1} A$ with $x_0 \in \sigma$, for all $y \in \sigma$, $(x_0 , y) \notin R$. 

On the other hand, by the assumption on $R$, there is a $y_0 \in A$ such that $(x_0, y_0) \in R$. By fineness of $\mu$, for $\mu$-measure one many elements $\sigma$ of $\mathcal{P}_{\omega_1} A$, both $x_0$ and $y_0$ are elements of $\sigma$.

Since $\mu$ is a filter, for $\mu$-measure one many elements $\sigma$ of $\mathcal{P}_{\omega_1} A$ with $x_0 \in \sigma$, for all $y \in \sigma$, $(x_0 , y) \notin R$ while both $x_0$ and $y_0$ are elements of $\sigma$ and $(x_0, y_0) \in R$. This gives us both $(x_0, y_0) \notin R$ and $(x_0, y_0 ) \in R$, a contradiction. This finishes the proof of the claim.
\qed
\end{proof}

We now know that for $\mu$-measure one many elements $\sigma$ of $\mathcal{P}_{\omega_1} A$, the following holds:
\begin{align*}
(\forall x \in \sigma ) \ ( \exists y \in \sigma) \ (x,y) \in R
\end{align*}

Let us pick such a $\sigma$. Then for any $x \in \sigma$, there is a $y \in \sigma$ such that $(x,y) \in R$. Since $\sigma$ is an element of $\mathcal{P}_{\omega_1} A$, it is countable, so we can fix a surjection $\pi \colon \omega \to \sigma$. Using this $\pi$, the above property of $\sigma$, and the well-orderedness of $(\omega , <)$, one can easily construct a desired $f \colon \omega \to A$. This finishes the proof of 1..

2. Suppose that there was an injection $i \colon \omega_1 \to 2^{\omega}$. We will derive a contradiction using supercompactness of $\omega_1$. For each $\alpha < \omega_1$, we write $x_{\alpha}$ for $i(\alpha)$. 

We first note that there is a non-principal $\sigma$-complete ultrafilter on $\omega_1$, i.e., $\omega_1$ is measurable. Since $\omega_1$ is supercompact, we can fix a fine normal measure $\mu$ on $\mathcal{P}_{\omega_1} \omega_1$. Let $\nu$ be as follows:
\begin{align*}
\nu = \{ A \subseteq \omega_1 \mid \text{ for $\mu$-measure one many elements $\sigma$ of $\mathcal{P}_{\omega_1} \omega_1$, $\sup \sigma \in A$} \}
\end{align*}
Then it is easy to see that $\nu$ is a non-principal $\sigma$-complete ultrafilter on $\omega_1$. 

Using this $\nu$, we will derive a contradiction as follows. Since $\nu$ is an ultrafilter on $\omega_1$, for any natural number $n$, there is an $k_n \in \{ 0, 1 \}$ such that the set $A_n = \{ \alpha < \omega_1 \mid x_{\alpha} (n) = k_n \}$ is of $\nu$-measure one. Since $\nu$ is $\sigma$-complete, the set $\displaystyle A = \bigcap_{n\in \omega} A_n$ is of $\nu$-measure one. By the property of each $A_n$, for any $\alpha$ in $A$, for all natural numbers $n$, $x_{\alpha} (n) = k_n$. But since $i$ is injective, $A$ has at most one element. This contradicts that $A$ is of $\nu$-measure one and $\nu$ is non-principal. This finishes the proof of 2.. This completes the proof of Theorem~\ref{thm:choice}.
\qed
\end{proof}

\begin{remark}
(2) of Theorem \ref{thm:choice} is the best one can hope for. ``$\omega_1$ is supercompact" does not imply ``there is no injection $f:\omega_2 \rightarrow \mathcal{P}(\omega_1)$". To see this, assume $\sf{ZFC}$ and there is a supercompact cardinal $\kappa$. Let $f: \kappa^+ \rightarrow \mathcal{P}(\kappa)$ be an injection in $V$. Let $\mathcal{T}$ be the Takeuti model defined at $\kappa$. Then clearly $f\in \mathcal{T}$ and in $\mathcal{T}$, $\kappa = \omega_1$ and $(\kappa^+)^V = \omega_2$.
\end{remark}

\section{Chang model and supercompactness of $\omega_1$}\label{sec:Chang model}

In this section, we prove Theorem~\ref{thm:infty-Borel}.
As a corollary, one can obtain usual regularity properties for sets of reals in the Chang model:
\begin{corollary}\label{cor:regularity}
Assume that $\omega_1$ is supercompact. Then every subset of $2^{\omega}$ in the Chang model is Lebesgue measurable and has the Baire property.
\end{corollary}

Corollary~\ref{cor:regularity} directly follows from Theorem~\ref{thm:choice}, Theorem~\ref{thm:infty-Borel}, and the following fact:
\begin{fact}[Essentially Solovay]\label{fact:regularity}
Assume that there is no injection from $\omega_1$ to $2^{\omega}$. Let $A$ be a subset of $2^{\omega}$ which is $\infty$-Borel. Then $A$ is Lebesgue measurable and has the Baire property.
\end{fact}

For the proof of Fact~\ref{fact:regularity}, one can refer to e.g., \cite[Theorem~2.4.2 \& Proposition~3.2.13]{Ikegami_PhD}.

To prove Theorem~\ref{thm:infty-Borel}, we prepare some definitions and lemmas. Given a set $X$, let $J(X)$ be the rudimental closure of $X \cup \{ X \}$.
\begin{definition}
Let $(C_{\alpha} \mid \alpha \in \text{Ord})$ be the following sequence:
\begin{enumerate}
\item $C_0 = \mathrm{L}_{\omega}$,

\item $C_{\alpha + 1} = J(C_{\alpha} \cup \alpha^{\omega})$, and

\item $\displaystyle C_{\beta} = \bigcup_{\alpha < \beta} C_{\alpha}$ when $\beta$ is a limit ordinal.
\end{enumerate}
Set $\displaystyle C = \bigcup_{\alpha \in \text{Ord}} C_{\alpha}$.
\end{definition} 

\begin{lemma}\label{lem:Chang1}
\begin{enumerate}

\item $C = \mathrm{L}(\bigcup_{\lambda\in Ord} \lambda^\omega)$.

\item $C_{\lambda}$ is in $\mathrm{L}(\lambda^{\omega})$.

\item For any set $A$ in the Chang model, there is an ordinal $\lambda$ such that $A$ is in $\mathrm{L}(\lambda^{\omega})$.
\end{enumerate}
\end{lemma}

\begin{proof}[Lemma~\ref{lem:Chang1}]

For 1., it is easy to see that $C$ is contained in the Chang model because the construction of the sequence $(C_{\alpha} \mid \alpha \in \text{Ord})$ is absolute between the Chang model and $V$. So it is enough to prove that $C$ contains the Chang model. For that it is enough to show that $C$ is an inner model of ZF containing all sets in $\text{Ord}^{\omega}$. By the construction of $(C_{\alpha} \mid \alpha \in \text{Ord})$, it is easy to see that $C$ contains all the sets in $\text{Ord}^{\omega}$, rudimentarily closed, satisfies Comprehension Scheme, and for any subset $X$ of $C$ in $V$, there is a set $Y$ in $C$ such that $X \subseteq Y$ (namely $C_{\alpha}$ for some big $\alpha$). Therefore, $C$ is an inner model of ZF containing all the sets in $\text{Ord}^{\omega}$, as desired.

For 2., it is enough to see that the construction of the sequence $(C_{\alpha} \mid \alpha \le \lambda)$ is absolute between $\mathrm{L}(\lambda^{\omega})$ and $V$, which follows by observing that $\lambda^{\omega}$ is in $\mathrm{L}(\lambda^{\omega})$.

For 3., let $A$ be any set in the Chang model. By 1., $A$ is in $C$ and hence there is an ordinal $\lambda$ such that $A$ is in $C_{\lambda}$. By 2., $C_{\lambda}$ is in $\mathrm{L}(\lambda^{\omega})$. Therefore, $A$ is in $\mathrm{L}(\lambda^{\omega})$, as desired. This completes the proof of Lemma~\ref{lem:Chang1}.
\qed
\end{proof}

By item 3. of Lemma~\ref{lem:Chang1}, to obtain Theorem~\ref{thm:infty-Borel}, it is enough to prove that for all $\lambda$, every subset of $2^{\omega}$ in $\mathrm{L}(\lambda^{\omega})$ is $\infty$-Borel. 

Throughout this section, we fix a non-zero ordinal $\lambda$ and a fine measure $\mu$ on $\mathcal{P}_{\omega_1} (\lambda^{\omega})$. We will show that every subset of $2^{\omega}$ in $\mathrm{L}(\lambda^{\omega})$ is $\infty$-Borel using $\mu$.

By Fact~\ref{fact:infty-Borel-char}, it is enough to show the following:

($\ast$) For any subset $A$ of $\lambda^{\omega}$ in $\mathrm{L}(\lambda^{\omega})$, there are some formula $\phi$ and a set $S$ of ordinals such that for all elements $x$ of $\lambda^{\omega}$, 
\begin{align}
x \in A \iff \mathrm{L}[S, x] \vDash \phi [S,x].\tag{$\dagger$}
\end{align}

If ($\dagger$) holds for all elements $x$ of $\lambda^{\omega}$, then we say that $A$ is \textbf{defined from the pair $(\phi , S)$} and we write $B_{(\phi , S)}$ for $A$.

The following is the key lemma in this section:
\begin{lemma}\label{lem:key lemma}
There is a function $F$ which is OD from $\mu$ such that 
\begin{enumerate}
\item $F$ is defined for all pairs $(\phi , S)$ where $\phi$ is a formula and $S$ is a set of ordinals,

\item $F(\phi, S)$ is of the form $(\psi , T)$ such that if $A$ is a subset of $(\lambda^{\omega})^{n+1}$ defined from $(\phi , S)$, then $\text{p}A = \{ \vec{x} \in (\lambda^{\omega})^n \mid (\exists y) \ (\vec{x},y) \in A\}$ is defined from $(\psi, T)$, i.e., if $A = B_{(\phi , S)}$, then $\text{p}A = B_{F(\phi, S)}$.
\end{enumerate}
\end{lemma}

To prove Lemma~\ref{lem:key lemma}, we use a variant of Vop$\check{\text{e}}$nka algebra: Let $S$ be a set of ordinals and $\sigma$ be an element of $\mathcal{P}_{\omega_1} \lambda^{\omega}$. We fix an injection $\iota \colon \text{OD}_{S, \sigma} \cap \mathcal{P}(\sigma) \to \text{HOD}_{S, \sigma}$ which is $\text{OD}$ from $S$ and $\sigma$ such that for all $t \in \lambda^{<\omega}$, $\iota (O_t) = t$ where $O_t = \{ x \in \sigma \mid t \subseteq x\}$. Set $B_V = \{ \iota (A) \mid A \in \text{OD}_{S,\sigma} \cap \mathcal{P}(\sigma) \}$. For $p,q \in B_V$, $p\le q$ if $\iota^{-1} (p) \subseteq \iota^{-1} (q)$. For an element $x$ of $\sigma$, set $G_x = \{ p \in B_V \mid x \in \iota^{-1}(p)\}$. 
\begin{fact}[Vop$\check{\text{e}}$nka]\label{fact:Vopenka}
\begin{enumerate}
\item In $\text{HOD}_{S,\sigma}$, $B_V$ is a complete Boolean algebra, and 

\item for any element $x$ of $\sigma$, $G_x$ is $B_V$-generic over $\text{HOD}_{S, \sigma}$, and $x$ is in $\mathrm{L}[S, B_V][G_x]$, which is a subclass of $\text{HOD}_{S,\sigma}[G_x]$.
\end{enumerate}
\end{fact}

Towards a proof of Lemma~\ref{lem:key lemma}, let us fix a fine measure $\mu$ on $\mathcal{P}_{\omega_1} \lambda^{\omega}$. For each $\sigma \in \lambda^{\omega}$, let $M_{\sigma} = \text{HOD}^{\mathrm{L}(S, \sigma)}_{S,\sigma}$ and $Q_{\sigma} = (B_V)^{M_{\sigma}}$. We will consider the ultraproducts $\prod_{\sigma} \mathrm{L}[S, Q_{\sigma}][x] / \mu $ for $x \in \lambda^{\omega}$. Using the finness of $\mu$, one can prove \L os' theorem for these ultraproducts (the proof is essentially the same as the one given in \cite[Lemma~2.3]{LRmu}). By DC, the above ultraproducts are all well-founded and we identify them with their transitive collapses. Set $S_{\infty} = \prod_{\sigma}S / \mu$ and $Q_{\infty} = \prod_{\sigma} Q_{\sigma} / \mu$. 

We are now ready to prove Lemma~\ref{lem:key lemma}.
\begin{proof}[Lemma~\ref{lem:key lemma}]

For simplicity, we will assume $n =1$ (the general case is treated in the same way). Let $A \subseteq (\lambda^{\omega})^2$ be defined from $(\phi, S)$, i.e., $A = B_{(\phi, S)}$. Then for all $x \in \lambda^{\omega}$,
\begin{align*}
& \ x \in \text{p}A\\
\iff & (\exists y \in \lambda^{\omega}) \ (x,y) \in B_{(\phi , S)}\\
\iff & \text{ for $\mu$-measure one many $\sigma$, } \mathrm{L}(S, \sigma) \vDash \lq\lq (\exists y) \ (x,y) \in B_{(\phi , S)}"\\
\iff & \text{ for $\mu$-measure one many $\sigma$, } \\
& \mathrm{L}[S, Q_{\sigma} , x] \vDash \lq\lq \Big(\exists p \in \text{Coll}\bigl(\omega, | \mathcal{P}(Q_{\sigma})| \bigr) \Bigr)  p\Vdash (\exists y) \ (\check{x},y) \in B_{(\phi , S)}"\\
\iff & \prod_{\sigma} \mathrm{L}[S, Q_{\sigma} , x]/\mu \vDash \lq\lq \Bigl( \exists p\in \text{Coll}\bigl(\omega, | \mathcal{P}(Q_{\infty})| \bigr) \Bigr) \ p\Vdash (\exists y) \ (\check{x},y) \in B_{(\phi , S_{\infty})}"\\
\iff & \mathrm{L}[S_{\infty}, Q_{\infty}, x] \vDash \lq\lq \Bigl( \exists p\in \text{Coll}\bigl(\omega, | \mathcal{P}(Q_{\infty})| \bigr) \Bigr) \ p\Vdash (\exists y) \ (\check{x},y) \in B_{(\phi , S_{\infty})}"
\end{align*}

The first equivalence follows from the assumption that $A$ is defined from $(\phi, S)$. The second equivalence follows from the fineness of $\mu$. The forward direction of the third equivalence follows from the property of the Vop$\check{\text{e}}$nka algebra $Q_{\sigma}$ given in Fact~\ref{fact:Vopenka}. The backward direction of the third equivalence follows from the fact that $\mathcal{P}(Q_{\sigma})^{\mathrm{L}[S, Q_{\sigma}, x]}$ is countable in $V$ because $Q_{\sigma}$ is countable by the fact that $M_{\sigma}$ is well-orderable in $V$ and $M_{\sigma}\cap \mathcal{P}(\sigma)$ is countable in $V$, and $\mathrm{L}[S, Q_{\sigma} , x]$ is a transitive model of \textsf{ZFC}. The fourth \& fifth equivalences follow from \L os' theorem for the ultraproduct $\prod_{\sigma} \mathrm{L}[S, Q_{\sigma}][x] / \mu $ and the definitons of $Q_{\infty}$ and $S_{\infty}$. 

Now let $T$ be the set of ordinals simply coding $S_{\infty}$ and $Q_{\infty}$, and $\psi$ be the formula stating $\lq\lq \Bigl( \exists p\in \text{Coll}\bigl(\omega, | \mathcal{P}(Q_{\infty})| \bigr) \Bigr) \ p\Vdash (\exists y) \ (\check{x},y) \in B_{(\phi , S_{\infty})}"$. Then $F(\phi , S) = (\psi , T)$ be as desired by the above equivalences. This completes the proof of Lemma~\ref{lem:key lemma}.
\qed
\end{proof}

We shall prove ($\ast$) above which gives us Theorem~\ref{thm:infty-Borel}. 
The idea is to look at the hierarchy $\bigl(\mathrm{L}_{\alpha} (\lambda^{\omega}) \mid \alpha \in \text{Ord}\bigr)$, and by induction on $\alpha$, to each definition of an element $A$ of $\mathrm{L}_{\alpha} (\lambda^{\omega})$, we assign certain $\phi$ and $S$ such that $A$ is defined from $(\phi , S)$. We fix an $F$ from Lemma~\ref{lem:key lemma}.

\begin{definition}
The hierarchy $\bigl(\mathrm{L}_{\alpha} (\lambda^{\omega}) \mid \alpha \in \text{Ord}\bigr)$ is defined as follows:
\begin{enumerate}
\item $\mathrm{L}_0 (\lambda^{\omega}) = \text{ the transitive closure of $\lambda^{\omega}$}$,

\item $\mathrm{L}_{\alpha +1} (\lambda^{\omega}) = \text{Def} \ \bigl( \mathrm{L}_{\alpha} (\lambda^{\omega}) , \in \bigr)$, and 

\item $\displaystyle \mathrm{L}_{\beta} (\lambda^{\omega}) = \bigcup_{\alpha < \beta} \mathrm{L}_{\alpha} (\lambda^{\omega})$ when $\beta$ is a limit ordinal.
\end{enumerate}
\end{definition}

\begin{remark}\label{rmk:key}
There is a sequence of partial surjections $\bigl( \pi_{\alpha} \colon \alpha^{<\omega} \times \lambda^{\omega} \to \mathrm{L}_{\alpha} (\lambda^{\omega}) \mid \alpha \in \text{Ord} \bigr)$ such that
\begin{enumerate}
\item if $\beta < \alpha$, then $\pi_{\beta} = \pi_{\alpha} \upharpoonright \beta^{<\omega} \times \lambda^{\omega}$, 

\item if $(\vec{\beta}, x ) \in \alpha^{<\omega} \times \lambda^{\omega}$, $\pi_{\alpha} (\vec{\beta}, x)$ is defined, and $\vec{\beta} = (\beta_0, \beta_1, \cdots , \beta_k)$, then $\pi_{\alpha}(\vec{\beta},x)$ is an element of $\mathrm{L}_{\beta_0 +1}(\lambda^{\omega})$ definable in the structure $\bigl(\mathrm{L}_{\beta_0} (\lambda^{\omega}), \in\bigr)$ via a formula coded by $\beta_1$ with some parameters of the form $\pi_{\beta_0} (\vec{\gamma}, y)$ where $\vec{\gamma}$ here depends only on $\vec{\beta}$, not on $x$, and

\item there is a $\vec{\beta} \in \alpha^{<\omega}$ such that for all $x \in \lambda^{\omega}$, $\pi_{\alpha} (\vec{\beta}, x) = x$.
\end{enumerate}
\end{remark}

\begin{definition}\label{key definition}
\begin{enumerate}
\item For an ordinal $\alpha$, let 
\begin{align*}
T^{\alpha} = \{ ( \phi , \vec{\beta}^0 , \cdots , \vec{\beta}^{n-1}, x_0, \cdots , x_{n-1}) \mid \mathrm{L}_{\alpha}(\lambda^{\omega}) \vDash \phi \bigl[ \pi_{\alpha} (\vec{\beta}^0 , x_0) , \cdots , \pi_{\alpha} (\vec{\beta}^{n-1}, x_{n-1})\bigr] \}
\end{align*}

\item For an ordinal $\alpha$, a formula $\phi$, and $\vec{\beta}^0 , \cdots , \vec{\beta}^{n-1} \in \alpha^{<\omega}$, let
\begin{align*}
T^{\alpha}_{(\phi, \vec{\beta}^0, \cdots , \vec{\beta}^{n-1})} = \{ (x_0, \cdots , x_{n-1}) \in (\lambda^{\omega})^n \mid (\phi , \vec{\beta}^0 , \cdots , \vec{\beta}^{n-1} , x_0, \cdots, x_{n-1} ) \in T^{\alpha} \}
\end{align*}
\end{enumerate}
\end{definition}

\begin{lemma}\label{lem:induction}
There is a function which is OD from $\mu$ sending $(\alpha ,p)$ to $q^{\alpha}_p = (\psi , S)$, where $\alpha$ is an ordinal and $p = (\phi , \vec{\beta}^0, \cdots , \vec{\beta}^{n-1})$ is as in Definition~\ref{key definition}, such that $T^{\alpha}_p$ is definable from $q^{\alpha}_p$.
\end{lemma}

\begin{proof}[Lemma~\ref{lem:induction}]

We prove the lemma by induction on $\alpha$. Let us fix $\alpha$. Then we prove the statement by induction on the complexity of $\phi$.

{\bf Case 1:} When $\phi$ is of the form $v \in w$ or $v =w$.

Let $\beta_* = \text{max} \{ \beta^0_0 , \beta^1_0 \} $. Then $\beta^* < \alpha$ and by Remark~\ref{rmk:key}, both $\pi_{\alpha} (\vec{\beta}^0, x_0)$ and $\pi_{\alpha} (\vec{\beta}^1 , x_1)$ are definable in the structure $\bigl(\mathrm{L}_{\beta_*} (\lambda^{\omega}), \in \bigr)$ with some parameters of the form $\pi_{\beta_*} (\vec{\gamma}, y)$ where $\vec{\gamma}$ here depends only on $\vec{\beta}^0$ and $\vec{\beta}^1$. Then one can find a formula $\phi'$ and some $\vec{\gamma}^0, \vec{\gamma}^1$ such that $T^{\alpha}_{\phi , \vec{\beta}^0, \vec{\beta}^1} = T^{\beta_*}_{\phi', \vec{\gamma}^0, \vec{\gamma}^1}$. By induction hypothesis, one can find a desired $q^{\alpha}_p$.

{\bf Case 2:} When $\phi$ is of the form $\neg \, \phi'$.

In this case, by induction hypothesis, letting $p' = (\phi', \vec{\beta}^0, \cdots , \vec{\beta}^{n-1})$, we have $q^{\alpha}_{p'} = (\psi , \vec{\gamma}^0 , \cdots , \vec{\gamma}^{n-1})$. Then $q^{\alpha}_p =  (\neg \, \psi , \vec{\gamma}^0 , \cdots , \vec{\gamma}^{n-1})$ is the desired one.

{\bf Case 3:} When $\phi$ is of the form $\phi_1 \wedge \phi_2$.

In this case, by induction hypothesis, letting $p_1 = (\phi_1, \vec{\beta}^0, \cdots , \vec{\beta}^{n-1})$ and $p_2 = (\phi_2, \vec{\beta}^0, \cdots , \vec{\beta}^{n-1})$, we have $q^{\alpha}_{p_1} = (\psi_1 , \vec{\gamma}^0, \cdots , \vec{\gamma}^{n-1})$ and $q^{\alpha}_{p_2} = (\psi_2 , \vec{\gamma}^0, \cdots , \vec{\gamma}^{n-1})$. Then $q^{\alpha}_p =  (\psi_1 \wedge \psi_2 , \vec{\gamma}^0 , \cdots , \vec{\gamma}^{n-1})$ is the desired one.

{\bf Case 4:} When $\phi$ is of the form $\exists v \, \phi'$.

In this case, by induction hypothesis, for each $\vec{\beta} \in \alpha^{<\omega}$, setting $p_{\vec{\beta}} = (\phi' ,\vec{\beta}, \vec{\beta}^0 , \cdots , \vec{\beta}^{n-1})$, we have $q^{\alpha}_{p_{\vec{\beta}}}$. We write $q_{\vec{\beta}}$ for $q^{\alpha}_{p_{\vec{\beta}}}$. Note that
\begin{align*}
& \ T^{\alpha}_{\phi , \vec{\beta}^0, \cdots , \vec{\beta}^{n-1}} \\
= & \{ (x_0 , \cdots x_{n-1} ) \mid \bigl(\exists y \in \mathrm{L}_{\alpha}(\lambda^{\omega})\bigr) \\
& \hspace{3cm} \mathrm{L}_{\alpha} (\lambda^{\omega}) \vDash \lq\lq \phi' [y , \pi_{\alpha}(\vec{\beta}^0, x_0) , \cdots , \pi_{\alpha} (\vec{\beta}^{n-1} , x_{n-1})]"\}\\
= & \{ (x_0, \cdots x_{n-1} ) \mid (\exists \vec{\beta} \in \alpha^{<\omega}) \ (\exists x \in \lambda^{\omega}) \\
& \hspace{3cm} \mathrm{L}_{\alpha} \vDash \lq\lq \phi' [ \pi_{\alpha} (\vec{\beta}, x), \pi_{\alpha}(\vec{\beta}^0, x_0) , \cdots , \pi_{\alpha} (\vec{\beta}^{n-1} , x_{n-1})]"\}\\
 = &  \bigcup_{x \in \lambda^{\omega}} \bigcup_{\vec{\beta} \in \alpha^{<\omega}} T^{\alpha}_{\phi' ,\vec{\beta}, \vec{\beta}^0, \cdots , \vec{\beta}^{n-1}}\\
= &  \bigcup_{x \in \lambda^{\omega}} \bigcup_{\vec{\beta} \in \alpha^{<\omega}}  B_{q_{\vec{\beta}}}\\
= & \bigcup_{x \in \lambda^{\omega}} B_{\bigvee_{\vec{\beta} \in \alpha^{<\omega}} q_{\vec{\beta}}}\\
= & B_{F\bigl( \bigvee_{\vec{\beta} \in \alpha^{<\omega}} q_{\vec{\beta}} \bigr)},
\end{align*}
where $B_{q_{\vec{\beta}}}$ is the subset of $\lambda^{\omega}$ defined from $q_{\vec{\beta}}$, $\bigvee_{\vec{\beta} \in \alpha^{<\omega}} q_{\vec{\beta}}$ is the pair $(\psi , S)$ defining the union $\bigcup_{\vec{\beta} \in \alpha^{<\omega} } B_{q_{\vec{\beta}}}$, and $F$ is from Lemma~\ref{lem:key lemma}. Therefore, $q^{\alpha}_p = F\bigl( \bigvee_{\vec{\beta} \in \alpha^{<\omega}} q_{\vec{\beta}} \bigr)$ is the desired one. This completes the proof of the lemma. 
\qed
\end{proof}

We are now ready to finish the proof of Theorem~\ref{thm:infty-Borel}.
\begin{proof}[Theorem~\ref{thm:infty-Borel}]

As in the paragraph after Fact~\ref{fact:infty-Borel-char}, it is enough to prove the following:
($\ast$) For any subset $A$ of $\lambda^{\omega}$ in $\mathrm{L}(\lambda^{\omega})$, there are some formula $\phi$ and a set $S$ of ordinals such that for all elements $x$ of $\lambda^{\omega}$, 
\begin{align}
x \in A \iff \mathrm{L}[S, x] \vDash \phi [S,x].\tag{$\dagger$}
\end{align}

Let $A$ be a subset of $\lambda^{\omega}$ in $\mathrm{L}(\lambda^{\omega})$. Then there is an ordinal $\alpha$ such that $A \in \mathrm{L}_{\alpha +1 } (\lambda^{\omega}) \setminus \mathrm{L}_{\alpha} (\lambda^{\omega})$. Let $\phi$ be a formula defining $A$ in the structure $\bigl(\mathrm{L}_{\alpha} (\lambda^{\omega}), \in \bigr)$ with some parameters $\pi_{\alpha}(\vec{\beta}^0 , x_0), \cdots , \pi_{\alpha} (\vec{\beta}^{n-1} , x_{n-1})$. Let $\vec{\beta} \in \alpha^{<\omega}$ be such that $\pi_{\alpha} (\vec{\beta} , x) = x$ for all $x \in \lambda^{\omega}$. Then 
\begin{align*}
A = & \{ x  \mid \mathrm{L}_{\alpha} (\lambda^{\omega}) \vDash \lq\lq \phi [\pi_{\alpha}(\vec{\beta}, x), \pi_{\alpha}(\vec{\beta}^0 , x_0), \cdots , \pi_{\alpha} (\vec{\beta}^{n-1} , x_{n-1})]" \} \\
 = & \{ x  \mid (x, x_0, \cdots , x_{n-1} ) \in T^{\alpha}_{\vec{\beta}, \vec{\beta}^0, \cdots , \vec{\beta}^{n-1}} \}\\
= & \{ x \mid (x, x_0 , \cdots , x_{n-1}) \in B_{q^{\alpha}_p}\},
\end{align*}
where $p = (\phi , \vec{\beta}, \vec{\beta}^0, \cdots , \vec{\beta}^{n-1})$ and $q^{\alpha}_p$ is from Lemma~\ref{lem:induction}. This shows that $A$ is defined from $q^{\alpha}_p$ with parameters $x_0, \cdots , x_{n-1}$, which easily gives us that $A$ is defined from $(\phi' , S)$ for some $\phi'$ and $S$. This finishes the proof of $(\ast)$, and hence this completes the proof of Theorem~\ref{thm:infty-Borel}.
\qed
\end{proof}

\section{Weak homogeneity and supercompactness of $\omega_1$}\label{sec:weakhomo}
In this section, we prove Theorem \ref{thm:weak_homo}. Recall the terminology about trees from Section \ref{sec:basic}. A tree $T$ is said to be on $\kappa$ if $T\subset (\omega\times \kappa)^{<\omega}$. For a tree $T$ on $\kappa$, for $s\in {}^{<\omega}\omega$, let $T_s = \{t \in {}^{lh(s)}\kappa : (s,t)\in T\}$. Let also $p[T] = \{f\in {}^{\omega}\kappa : \exists x \forall n (x\restriction n, f\restriction n)\in T \}$. Every tree $T$ considered in the following will be on $\kappa$ for some $\kappa$. 

Following \cite{martin2008weakly}, we define what it means for a tree $T$ on $\kappa$ to be \textbf{weakly homogeneous}. First, for $n<\omega$, $\kappa$ an infinite cardinal, $\lambda$ a nonzero ordinal, let MEAS$^{\kappa,\lambda}_n$ be the set of all $\kappa$-complete measures on ${}^n\lambda$. For $m<n <\omega$, for $X\subseteq \lambda^m$, let $ext_n(X) = \{t \in \lambda^n : t\restriction m\in X \}$. A \textbf{$\lambda$-tower of measures} is a sequence $(\mu_n : n<\omega)$ such that
\begin{enumerate}[(i)]
\item for each $n$, $\mu_n \in $ MEAS$^{\omega_1,\lambda}_n$, and
\item for $m<n$, $\mu_m = proj_m(\mu_n)$, where $proj_m(\mu_n) = \{X\subseteq \lambda^m : ext_n(X)\in \mu_n\}$.
\end{enumerate}

A tower $(\mu_n : n<\omega)$ is \textbf{countably complete} if for every sequence $(X_n : n<\omega)$ such that $X_n\in \mu_n$ for all $n<\omega$, there is a function $f:\omega \rightarrow \lambda$ such that $f\restriction n\in X_n$ for all $n$.

\begin{definition}\label{def:weak_homo}
Let $T$ be a tree on $\lambda$. $T$ is \textbf{weakly homogeneous} if there is a sequence $(M_s : s\in {}^{<\omega}\omega)$ such that
\begin{enumerate}[(i)]
\item for each $s$, $M_s$ is a countable subset of MEAS$^{\omega_1,\lambda}_{lh(s)}$ and for each $\mu \in M_s$, $T_s\in \mu_s$.
\item for all $x\in p[T]$, there is a countably complete $\lambda$-tower of measures $(\mu_n : n<\omega)$ such that for each $n$, $\mu_n \in M_{x\restriction n}$. 
\end{enumerate}
\end{definition}

\begin{proof}[Theorem \ref{thm:weak_homo}]
Let $T$ be a tree on $\lambda$. \cite{martin2008weakly} shows that $T$ is weakly homogeneous provided the following conditions hold:
\begin{enumerate}[(A)]
\item There is a countably complete, normal fine measure on $\mathcal{P}_{\omega_1}(\bigcup_n (\mathcal{P}({}^n\lambda)\cup \rm{MEAS}$$^{\omega_1,\lambda}_n))$.
\item The Axiom of Dependent Choice holds for relations on $\mathcal{P}(\lambda)$.
\item There is a wellorder on $\bigcup_n \rm{MEAS}$$^{\omega_1,\lambda}_n$.
\end{enumerate}

We need to verify (A), (B), (C) follow from the supercompactness of $\omega_1$. (A) is obvious. (B) follows from Theorem \ref{thm:choice}. Now we verify (C). Let $X = \bigcup_n \rm{MEAS}$$^{\omega_1,\lambda}_n$. We need to show that $X$ is wellorderable.

It is enough to prove that MEAS$^{\omega_1,\lambda}_1$ is well-orderable. This is because for each $n > 1$, there is a bijection from MEAS$^{\omega_1,\lambda}_1$ onto MEAS$^{\omega_1,\lambda}_n$. Such a bijection is induced by a bijection between $\lambda$ and $\lambda^n$. Hence, MEAS$^{\omega_1,\lambda}_n$ is well-orderable. By $\sf{DC}$, $X$ is well-orderable. 

Let  $Z = \mathcal{P}(\lambda)\cup \rm{MEAS}$$^{\omega_1,\lambda}_1$. Let $U$ be a countably complete, normal fine measure on $\mathcal{P}_{\omega_1}(Z)$. Given $\mu \in $MEAS$^{\omega_1,\lambda}_1$ and $\sigma \in \mathcal{P}_{\omega_1}(Z)$, let 
\begin{center}
$f_\mu(\sigma) = min \bigcap (\sigma\cap \mu)$.
\end{center}
So $f_\mu$ is a function from $\mathcal{P}_{\omega_1}(Z)$ into the ordinals. 

\begin{claim}
Suppose $\mu\neq \nu$ are in MEAS$^{\omega_1,\lambda}_1$. Then $\forall^*_\sigma U \ f_\mu(\sigma) \neq f_\nu(\sigma)$, here ``$\forall^*_\sigma U \varphi(\sigma)$" abbreviates the statement ``the set of $\sigma$ such that $\varphi(\sigma)$ is in $U$".
\end{claim}
\begin{proof}
Let $A$ witness $\mu\neq \nu$. Without loss of generality, assume $A\in \mu$ and $\neg A \in \nu$. By fineness of $U$, $\forall^*_U \sigma$, $\{A,\neg A\}\subset \sigma$. Fix such a $\sigma$. Then $f_\mu(\sigma) \in A$ and $f_\nu(\sigma) \in \neg A$. Since $A,\neg A$ are disjoint, $f_\mu(\sigma)\neq f_\nu(\sigma).$ \qed
\end{proof}
Let $\pi: X \rightarrow \prod_{\sigma\in \mathcal{P}_{\omega_1}(Z)} Ord \slash U$ be defined as: $\pi(\mu) = [f_\mu]_U$. The claim gives us that $\pi$ is an injection. By $\sf{DC}$, $\prod_{\sigma\in \mathcal{P}_{\omega_1}(Z)} Ord \slash U$ is well-founded and furthermore is well-ordered. Therefore, $X$ is well-ordered as desired. \qed
\end{proof}

\section{$\sf{AD}^+$, $\sf{AD}_{\mathbb{R}}$, and supercompactness of $\omega_1$}\label{sec:equiv}

In this section, we prove Theorem \ref{thm:equiv}. We need the following important notion of an \textbf{envelope of a pointclass}. This was first formulated by D.A. Martin (cf. \cite{jackson2010structural}). We only need the notion of an envelope of an inductive-like pointclass.

\begin{definition}\label{def:envelope}
A pointclass $\Gamma$ is \textbf{inductive-like} if it is closed under continuous reductions and real quantifiers. \footnote{Here and below, $\Gamma$ is always a boldface pointclass.}

Let $\Gamma$ be an inductive-like pointclass. Let $\check{\Gamma}$ be the pointclass consisting of all $A$ such that $\neg A\in \Gamma$. Let $\Delta_\Gamma = \Gamma\cap \check{\Gamma}$. \textbf{The envelope of $\Gamma$}, $Env(\Gamma)$, is the pointclass consisting of all $A$ such that for any countable $\sigma\subset \mathbb{R}$, there is some $A'\in \Delta_\Gamma$ such that $A\cap \sigma = A'\cap \sigma$.
\end{definition}

The following fact about envelopes is crucial for our argument.
It is essentially proved in 
\cite{wilson2012contributions} (which deals with generic large cardinal properties of $\omega_1$ in $\mathsf{ZFC}$ rather than with large cardinal properties of $\omega_1$, but the argument carries over to the present context.)

\begin{lemma}[Wilson]\label{lemma:scale-from-strong-compactness}
 Assume $\mathsf{ZF} + \mathsf{DC}$. Let $\Gamma$ be an inductive-like pointclass with the scale property.
 Suppose that $\omega_1$ is $Env(\Gamma)$-strongly compact.
 Then there is a scale on a  universal $\check{\Gamma}$ set, each of whose prewellorderings is in $Env(\Gamma)$.
\end{lemma}

In the above, $Env(\Gamma)$ is a self-dual pointclass (closed under complementation) containing $\Gamma$ (cf. \cite{jackson2010structural} and \cite{wilson2012contributions}). So $\Gamma\cup \check{\Gamma}\subset Env(\Gamma)$. Another important fact about envelopes is that if $\mathsf{ZF} + \mathsf{DC}_\mathbb{R}$ holds and the boldface ambiguous part $\Delta_\Gamma$ of the pointclass $\Gamma$ is determined, as it is here, then $Env(\Gamma)$ is determined and projectively closed (Wilson \cite{wilson2012contributions}; based on work of Kechris, Woodin, and Martin.)
Therefore Wadge's lemma applies to it, as one can easily verify that the relevant games are determined.
 Moreover, the Wadge preordering\footnote{We are abusing notation here; really it is a preordering of pairs $\{B,\neg B\}$ where $B \in Env(\Gamma)$.} of $Env(\Gamma)$ is a prewellordering:
 otherwise by $\mathsf{DC}_\mathbb{R}$ we could choose a sequence of pointsets in $Env(\Gamma)$ that was strictly decreasing in the Wadge ordering, but then by the proof of the Martin--Monk theorem we get a contradiction.
 (Again one can easily verify that the relevant games are determined.)

Note that the prewellorderings of a scale as in Lemma \ref{lemma:scale-from-strong-compactness} must be Wadge-cofinal in $Env(\Gamma)$; otherwise the sequence of prewellorderings itself would be coded by a set of reals in $Env(\Gamma)$, which is impossible.

The following fundamental fact about $\sf{AD}^+$ is due to H.W. Woodin (cf. \cite{ketchersid2011more}).
\begin{theorem}[Woodin]\label{thm:ADSuslin}
The following are equivalent.
\begin{enumerate}
\item $\textsf{AD}^+$.
\item $\textsf{AD} + $ the class of Suslin cardinals is closed.
\end{enumerate}
\end{theorem}

We will also need the following results due to D.A. Martin and Woodin.
\begin{theorem}\label{thm:ADRSuslin}
Assume $\textsf{ZF + DC}$. The following are equivalent.
\begin{enumerate}
\item $\textsf{AD}_{\mathbb{R}}$.
\item $\textsf{AD}^+ + $ every set is Suslin.
\end{enumerate}
\end{theorem}

\begin{proof}[Theorem \ref{thm:equiv}]
First, note that by Theorem \ref{thm:choice}, $\sf{DC}$ follows from supercompactness of $\omega_1$. The $(\Leftarrow)$ direction follows immediately from Theorem \ref{thm:ADRSuslin}. For the $(\Rightarrow)$ direction, suppose $\sf{AD}_\mathbb{R}$ fails. By Theorems \ref{thm:ADRSuslin} and \ref{thm:ADSuslin}, there is the largest Suslin cardinal $\kappa < \Theta$. Let $\Gamma = S(\kappa)$ be the pointclass of $\kappa$-Suslin sets, i.e. the set of $A$ such that there is a tree $T$ on $\kappa$ such that $A = p[T]$. Then by the basic theory of pointclasses under determinacy, $\Gamma$ is inductive-like, has the scale property, and furthermore, there is some $B\in \check{\Gamma}$ that is not Suslin (equivalently, there is no scale on $B$). Such a $B$ can be taken to be a universal $\check{\Gamma}$-set. The reason $B$ cannot have a scale is because otherwise, $B$ is Suslin. Since $\kappa$ is the largest Suslin cardinal, $B$ is $\kappa$-Suslin (i.e. there is a tree $T$ on $\kappa$ such that $B = p[T]$). But this means $B\in \Gamma$, which cannot happen.

By Lemma \ref{lemma:scale-from-strong-compactness}, $\sf{ZF+DC} + \omega_1$ is supercompact implies that every set in $\check{\Gamma}$ has a scale, each of whose prewellorderings is in $Env(\Gamma)$.\footnote{In fact, \cite[Section 5]{trang2016determinacy} just needs $\omega_1$ is $Env(\Gamma)$-strongly compact for this.} This is a contradiction.\footnote{An alternative proof for the $(\Rightarrow)$ direction is as follows. Suppose $\sf{AD}_\mathbb{R}$ fails. Let $\kappa<\Theta$ be the largest Suslin cardinal and $\Gamma = S(\kappa)$. Let $A$ be a universal $\Gamma$-set. Then by results of Section \ref{sec:weakhomo}, $A$ is weakly homogeneously Suslin. By the Martin-Solovay construction, $\neg A$ is Suslin. But $\neg A \in \check{\Gamma}\backslash \Gamma$. This contradicts the fact that $\Gamma$ is the largest Suslin pointclass.} \qed

\end{proof}

The following may be a more approachable version of a well-known conjecture that $\sf{AD}$ is equivalent to $\sf{AD}^+$.

\begin{conjecture}\label{conj:AD}
Assume $\omega_1$ is supercompact. $\sf{AD}$ is equivalent to $\sf{AD}^+$.
\end{conjecture}

\section{$\sf{HPC}$ and supercompactness of $\omega_1$}\label{sec:HPC}

In this section, we will prove Theorem \ref{thm:det}. First, we note that we do not need the full ``$\omega_1$ is supercompact" hypothesis in the proof of the theorem; one just needs:
\begin{itemize}
\item $\omega_1$ is $\mathbb{R}$-supercompact, and
\item $\neg \square_{\omega_1}$.
\end{itemize}
Both of these are consequences of $\omega_1$ is supercompact, cf\cite[Section 1]{trang2016determinacy}.

Now we explain \textit{Hod Pair Capturing} ($\sf{HPC}$). This hypothesis and the notion of \textit{least branch hod pair} (lbr hod pair) are formulated by John Steel. The reader can see \cite{normalization_comparison} for a detailed discussion regarding topics concerning least-branch hod premice, lbr hod pairs, and $\sf{HPC}$. The main thing one needs from $\sf{HPC}$ are the facts given by Theorem \ref{thm:lbrScS}. For basic facts about inner model theoretic notions such as iteration strategies, see \cite{steel2010outline}.
\begin{definition}[lbr hod pair, \cite{normalization_comparison}]
$(\mathcal{P},\Sigma)$ is an lbr hod pair if $\mathcal{P}$ is an lpm (least-branch hod premouse) and $\Sigma$ is a complete strategy for $\mathcal{P}$ that is normalizes well and has strong hull condensation.
\end{definition}

\begin{definition}[$\sf{HPC}$, \cite{normalization_comparison}]
Suppose $A$ is Suslin co-Suslin. Then there is an lbr hod pair $(\mathcal{P},\Sigma)$ such that $A$ is Wadge reducible to $Code(\Sigma)$.
\end{definition}

It is conjectured that $\sf{AD}^+$ implies $\sf{HPC}$. $\sf{HPC}$ and its variations have been shown to hold in very strong models of determinacy, cf. \cite{hod_mice_LSA}.

In the above, a complete strategy acts on all countable stacks of countable normal trees. The reader can consult \cite{normalization_comparison} for more details on lbr hod pairs. The basic theory of lbr hod pairs has been worked out in \cite{normalization_comparison}. What we need are a couple of facts about them. In the following, we fix a canonical coding $Code$ of subsets of $HC$ by subsets of $\mathbb{R}$.\footnote{One way of defining $Code$ is as follows. Let $\pi: \mathbb{R}\rightarrow HC$ be a surjection defined as: for any $x$ that codes a well-founded relation $E_x$ on $\omega$, let $\pi(x)$ be the transitive collapge of the structure $(\omega,E_x)$. Then for any $A\subseteq HC$, $Code(A)$ is defined to be $\pi^{-1}[A]$.} Given an lbr hod pair $(\mathcal{P},\Sigma)$, for $n < \omega$, $\mathcal{M}_n^{\Sigma,\sharp}$ is the minimal, active $\Sigma$-mouse that has $n$ Woodin cardinals. See for instance \cite{scales_hybrid_mice} for a precise definition. The following facts are relevant for us. 

\begin{lemma}\label{lem:term}

Let $(\mathcal{P},\Sigma)$ be an lbr hod pair. Let $M = \mathcal{M}_n^{\Sigma,\sharp}$ and $\Lambda$ be $M$'s canonical strategy. Let $\lambda$ be the largest Woodin cardinal of $M$. There is a term $\tau_\Sigma\in M^{Coll(\omega,\lambda)}$ such that whenever $i:M\rightarrow N$ is an iteration embedding via an iteration according to $\lambda$, and $g\subseteq Coll(\omega, i(\lambda))$ is $N$-generic, then
\begin{center}$[i(\tau_\Sigma)]_g = Code(\Sigma)\cap N[g]$.\end{center}

\end{lemma}

\begin{theorem}\label{thm:lbrScS}
Suppose $\omega_1$ is $\mathbb{R}$-supercompact, $\neg \square_{\omega_1}$. Suppose $(\mathcal{P},\Sigma)$ is an lbr hod pair. Then 
\begin{enumerate}
\item (\cite[Section 2]{mouse_pairs}) $Code(\Sigma)$ is Suslin co-Suslin.
\item (\cite[Section 3]{trang2016determinacy}) $\mathcal{M}_2^{\Sigma,\sharp}$ exists.
\end{enumerate}
\end{theorem}

\begin{remark}
We note that in the above theorem, the hypothesis $\omega_1$ is $\mathbb{R}$-supercompact is used in (1) to extend $\Sigma$ to act on all stacks $\vec{\mathcal{W}}$ such that there is a surjection of $\mathbb{R}$ onto $\vec{\mathcal{W}}$. The proof of (2) just needs $\omega_1$ is $\mathbb{R}$-strongly compact and $\neg\square_{\omega_1}$.
\end{remark}

The following theorem, due to Neeman, is our main tool for proving determinacy.

\begin{theorem}[Neeman, \cite{neeman1995optimal}]\label{thm:termCap}
Suppose $A\subseteq \mathbb{R}$. Suppose $(M,\Lambda, \delta)$ is such that 
\begin{enumerate}
\item $M \vDash \delta$ is Woodin;
\item $\Lambda$ is an $\omega_1+1$-iteration strategy for $M$;
\item there is a term $\tau \in M^{Coll(\omega,\delta)}$ such that whenever $i: M \rightarrow N$ is an iteration map according to $\Lambda$, $g\subseteq Coll(\omega,i(\delta))$ is $N$-generic, then $i(\tau)_g = A\cap N[g]$.
\end{enumerate}
Then $A$ is determined.
\end{theorem}

\begin{proof}[Theorem \ref{thm:det}]
Let $A$ be Suslin co-Suslin. By $\sf{HPC}$, let $(\mathcal{P},\Sigma)$ be an lbr hod pair such that $A$ is Wadge reducible to $Code(\Sigma)$. Let $x\in \mathbb{R}$ witness this; we let $\sigma_x$ be the continuous function given by $x$ such that $\sigma_x^{-1}[Code(\Sigma)] = A$.\footnote{We can construe $x$ as a function ${}^{<\omega}\omega \rightarrow {}^{<\omega}\omega$ and this naturally gives rise to a continuous (in fact Lipschitz) function $\sigma_x: {}^{\omega}\omega \rightarrow {}^{\omega}\omega$.} By Theorem \ref{thm:lbrScS}, $Code(\Sigma)$ is Suslin co-Suslin and $\mathcal{M}_2^{\Sigma,\sharp}$ exists. Let $\Lambda$ be the canonical iteration strategy for $\mathcal{M}_2^{\Sigma,\sharp}$ and $\delta_0<\delta_1$ be the Woodin cardinals of $\mathcal{M}_2^{\Sigma,\sharp}$. Let $\mathcal{T}$ be an iteration tree with the following properties:
\begin{itemize}
\item $\mathcal{T}$ is according to $\Lambda$.
\item Letting $i: \mathcal{M}_2^{\Sigma,\sharp} \rightarrow \mathcal{N}$ be the corresponding iteration embedding, then $x$ is $\mathcal{N}$-generic for the extender algebra at $i(\delta_0)$.\footnote{The reader can see \cite{steel2010outline} for more detailed discussions about the extender algebra and genericity iteration trees. All we need here is the fact there is some forcing $\mathbb{Q}\in \mathcal{N}$ such that card$^\mathcal{N}(\mathbb{Q}) \leq i(\delta_1)$ and such that there is some $\mathcal{N}$-generic filter $h\subset \mathbb{Q}$ such that $x\in \mathcal{N}[h]$.}
\end{itemize}
Now we can construe $\mathcal{N}[x]$ as a $\Sigma$-mouse over $x$, which we will call $\mathcal{M}$. Note that $i(\delta_1)$ is a Woodin cardinal of $\mathcal{M}$ and $\Lambda$ induces a strategy $\Psi$ on $\mathcal{M}$.

We note that $(\mathcal{M},\Psi, i(\delta_1))$ satisfies the hypothesis of Theorem \ref{thm:termCap} for $A$. Let $\tau_\Sigma$ is given as in Lemma \ref{lem:term} for $(\mathcal{M}_2^{\Sigma,\sharp},\Lambda,\delta_1)$, then $i(\tau_\Sigma)$ induces a term $\sigma_\Sigma\in M^{Coll(\omega,i(\delta_1))}$ satisfying (3) of Theorem \ref{thm:termCap}. Let $\mathbb{P} = Coll(\omega,i(\delta_1))$. The term $\tau$ consists of $(1_{\mathbb{P}},\sigma)$ where $1_{\mathbb{P}}\Vdash_{\mathbb{P}} ``\sigma$ is a real and  $\sigma \in \sigma_x^{-1}[\sigma_\Sigma]$". 

By Theorem \ref{thm:termCap}, $A$ is determined. This completes the proof of Theorem \ref{thm:det}. \qed
\end{proof}

We conjecture that $\sf{HPC}$ is not needed in Theorem \ref{thm:det}. \cite{trang2016determinacy} has shown that $\omega_1$ is supercompact implies that all sets in $L(\mathbb{R})$ are Suslin co-Suslin and are determined and much more.\footnote{\cite{trang2016determinacy} shows that there is a transitive class model $M$ containing all reals such that $M \vDash \sf{AD}_\mathbb{R} + \sf{DC}$.} One may hope to prove Conjecture \ref{conj:ScSdet} by showing that every Suslin co-Suslin set is homogeneously Suslin. Theorem \ref{thm:weak_homo} shows that every Suslin co-Suslin set is a projection of a homogenously Suslin, hence determined, set.

\begin{conjecture} \label{conj:ScSdet}
Assume $\omega_1$ is supercompact. For any Suslin co-Suslin set $A$, $A$ is determined.
\end{conjecture}

%
%
%
%
\bibliographystyle{spmpsci}
\bibliography{omega1spct.bib}

\end{document}